\documentclass[reqno,a4paper,11pt]{amsart}
\usepackage[all,poly]{xy}
\usepackage{amsfonts}
\usepackage[mathcal]{eucal}
\usepackage{amssymb}
\usepackage{amsmath}
\usepackage{mathrsfs}
\usepackage{color}
\usepackage[pagebackref,colorlinks]{hyperref}
\usepackage{enumerate}
\usepackage{comment}
\usepackage{hhline}
\usepackage{graphics}
\usepackage{float}
\usepackage{comment}

\usepackage{enumerate}

\theoremstyle{plain}
\newtheorem {lemma}{Lemma}
\newtheorem {proposition}[lemma]{Proposition}
\newtheorem {theorem}[lemma]{Theorem}
\newtheorem {corollary}[lemma]{Corollary}

\theoremstyle{definition}
\newtheorem{definition}[lemma]{Definition}

\newtheorem{remark}[lemma]{Remark}

\newtheorem {example}[lemma]{Example}

\newcommand{\N}{\mathbb{N}}
\newcommand{\X}{\langle X\rangle}

\newcommand{\SN}{\operatorname{SN}}

\def\Oo{\mathcal{O}}

\newcommand{\irr}{\operatorname{irr}}
\newcommand{\NF}{\operatorname{NF}}
\newcommand{\st}{\operatorname{st}}

\newcommand{\ann}{\operatorname{ann}}

\newcommand{\supp}{\operatorname{supp}}

\newcommand{\bemph}[1]{{\upshape#1}} % define how emphasised brackets should look
\newcommand{\ep}[1]{\bemph{(}#1\bemph{)}} % parentheses

\title[Normal forms for weighted
Leavitt path algebras]{Applications of normal forms for weighted
Leavitt path algebras: simple rings and domains}

\author{Roozbeh Hazrat}
\address{Centre for Research in Mathematics,
Western Sydney University,
Australia}\email{r.hazrat@westernsydney.edu.au}

\author{Raimund Preusser}
\address{Department of Mathematics,
University of Brasilia, Brazil}
\email{raimund.preusser@gmx.de}
\date{}

\thanks{The first author would like to acknowledge Australian Research Council grants DP150101598 and DP160101481. A part of this work was done at the University of M\"unster, where the first author was a Humboldt Fellow.}


\begin{document}
\maketitle

\begin{abstract} Weighted Leavitt path algebras (wLpas) are a generalisation of Leavitt path algebras (with graphs of weight $1$) and cover the algebras $L_K(n,n+k)$ constructed by Leavitt. Using Bergman's diamond lemma, we give normal forms for elements of a weighted Leavitt path algebra. This allows us to produce a basis for a wLpa. 
%When the weight is one, we get a basis for Leavitt path algebras as described in~\cite{zel12}. 
Using the normal form we classify the wLpas which are domains, simple and graded simple rings. For a large class of weighted Leavitt path algebras we establish a local valuation and as a consequence we prove that these algebras are prime, semiprimitive and nonsingular but contrary to Leavitt path algebras, they are not graded von Neumann regular.
\end{abstract}

%\tableofcontents

\section{Introduction}

%In this paper ``ring" always means associative, unital ring. By an $R$-module we mean a left $R$-module.

In a series of papers William Leavitt studied algebras that are now denoted by $L_K(n,n+k)$ and have been coined Leavitt algebras. Let $X=(x_{ij})$ and $Y=(y_{ji})$ be $n\times (n+k)$ and $(n+k)\times n$ matrices consisting of symbols $x_{ij}$ and 
$y_{ji}$, respectively. Then for a field $K$, $L_K(n,n+k)$ is a $K$-algebra generated by all $x_{ij}$ and $y_{ji}$ subject to the relations $XY=I_{n+k}$ and $YX=I_n$. In~\cite[p.190]{vitt56} Leavitt studied these algebras for  $n=2$ and $k=1$, in~\cite[p.322]{vitt57} for any $n\geq 2 $ and $k=1$ and finally in~\cite[p.130]{vitt62}  for arbitrary $n$ and $k$. He established that these algebras are of type $(n,k)$. He further showed that   
$L_K(1,k+1)$ are simple rings and $L_K(n,n+k)$, $n\geq 2$ are domains. Recall that a ring $A$ is of type $(n,k)$ if $n$ and $k$ are the least positive integers such that $A^n\cong A^{n+k}$ as left $A$-modules.  He proved these statements by formulating a normal form for the elements of his algebras. This normal form was worked out more systematically by P.M. Cohn in~\cite{cohn66} who showed that $L_K(n,n+k)$ is a domain using a trace method. The normal forms for algebras defined by  generators and relations were streamlined by G. Bergman in his influential paper~\cite{bergman78}, called the diamond lemma, following the paper~\cite{newman42}. 
 
 Leavitt path algebras were introduced a decade ago~\cite{aap05,Ara_Moreno_Pardo}, associating a $K$-algebra to a directed graph. For a graph with one vertex and $k+1$ loops, it recovers the Leavitt algebra $L_K(1,k+1)$.  The definition and the development of the theory were  inspired on the one hand by Leavitt's construction of  $L_K(1,k+1)$ and on the other hand by Cuntz algebras $\Oo_n$~\cite{cuntz1} and  Cuntz-Krieger algebras in $C^*$-algebra theory~\cite{raeburn}.  The Cuntz algebras and later Cuntz-Krieger type $C^*$-algebras revolutionised $C^*$-theory, leading ultimately to the astounding
Kirchberg-Phillips classification theorem~\cite{phillips}. In the last decade the Leavitt path algebras have created the same type of stir in the algebraic community. The development of Leavitt path algebras and its interaction with graph $C^*$-algebras have been well-documented in several publications and we refer the reader to~\cite{abrams-ara-molina} and the references therein. 
 
 Since their introductions, there have been several attempts to introduce a generalisation of Leavitt path algebras which would cover the algebras $L_K(n,n+k)$ for any $n\geq 1$, as well. Ara and Goodearl's Leavitt path algebras of separated graphs were introduced in~\cite{aragoodearl} which gives $L_K(n,n+k)$ as a corner ring of some separated graphs. The  weighted Leavitt path algebras were introduced in~\cite{hazrat13} which gives $L_K(n,n+k)$ for a weighted graph with one vertex and $n+k$ loops of weight $n$. If the weights of all the edges are $1$ (i.e., the graph is unweighted), then the weighted Leavitt path algebras reduce to the usual Leavitt path algebras. The structure of weighted Leavitt path algebras remained to be explored. In this paper we take a step in this direction (no one had looked at the topic systematically so far).
 
 In Section \ref{sec2} we develop systematically a normal form for elements of weighted Leavitt path algebras by using Bergman's diamond machinery. This allows us to describe a basis for such algebras. In turn we can then characterise simple and graded simple weighted Leavitt path algebras (cf. Section \ref{sec3}). There are unexpected interesting cases. For example, for the weighted graphs $E$ and $F$  below with one edge of weight two and the rest of weight one, the weighted Leavitt path algebra $L(E,\omega)$ is simple, whereas $L(F,w)$ is not ($\mathbb Z^2$-graded) simple. 
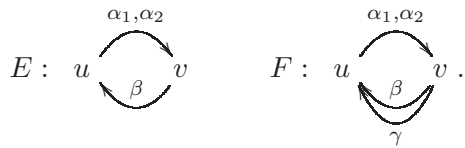
\[
\xymatrix{
E:  \!\!\!\!\!\!\!\!&\!\!\!\! u  \ar@/^1.2pc/[r]^{\alpha_1,\alpha_2}  & v  \ar@/^1.2pc/[l]_{\beta}
& F:  \!\!\!\!\!\!\!\!&\!\!\!\! u  \ar@/^1.2pc/[r]^{\alpha_1,\alpha_2}  & v \ar@/^1.7pc/[l]^{\gamma}  \ar@/^1.2pc/[l]_{\beta}
}.
\] 
In Theorem~\ref{thmn1} we show that a simple weighted Leavitt path algebra is isomorphic to a Leavitt path algebra. In Section \ref{sec7} we construct a local valuation for a large class of weighted Leavitt path algebras (so-called LV-algebras). Using the valuation we show these algebras are prime, semiprimitive and nonsingular but contrary to Leavitt path algebras, they are not graded von Neumann regular. Further we classify the weighted Leavitt path algebras which are domains (see Theorem \ref{cor1}).
 This allows us to obtain a much larger class of prime and nonsingular rings than Leavitt path algebras.

We finish this introduction by mentioning that K. McClanahan \cite{mccal1,mccal2} studied $U^{nc}_{n,n+k}$-algebras (first considered by D.V. Voiculescu). These are $C^*$-algebras generated by elements $u_{ij}$, $1\leq i \leq n$, $1\leq j\leq n+k$ subject to the relations $uu^*=I_n$ and 
$u^* u= I_{n+k}$, where $u=(u_{ij})_{n\times (n+k)}$. Note that  the Cuntz algebra $\Oo_n$ corresponds to $U^{nc}_{1,n}$.  Clearly in the pure algebra setting, $U^{nc}_{n,n+k}$ corresponds to the Leavitt algebra $L_{\mathbb C}(n,n+k)$. However, the concept of weighted graph $C^*$-algebra which as a special case cover $U^{nc}_{n,n+k}$  is yet to be defined and explored. 
 
\section{Normal forms for weighted Leavitt path algebras}\label{sec2}

We begin this section by recalling the concept of weighted graphs and weighted Leavitt path algebras, first introduced in~\cite{hazrat13}. Throughout the semigroup of positive integers is denoted by $\mathbb N$ and the monoid of non-negative integers by $\mathbb N_0$.

\begin{definition}[{\sc Weighted graph}]\label{def1}
A {\it weighted graph} $E=(E^0,E^{\st},E^1, s, r, \omega)$ consists of three countable sets, $E^0$ called {\it vertices}, $E^{\st}$ {\it structured edges} and $E^1$ {\it edges}, maps $s, r: E^{\st}\rightarrow E^0$, and a {\it weight map} $\omega: E^{\st}\rightarrow \N$ such that
\[E^1 = \bigsqcup\limits_{\alpha\in E^{\st}}\{\alpha_i \ | \ 1 \leq i \leq \omega(\alpha)\},\]
i.e., for any $\alpha\in E^{\st}$, with $\omega(\alpha) = k$, there are $k$ distinct elements $\{\alpha_1, . . . , \alpha_k\}$, and $E^1$ is the disjoint union of all such sets for all $\alpha\in E^{\st}$.

\end{definition}

\begin{remark}
We sometimes write $(E,\omega)$ to emphasise the graph is weighted. If $s^{-1}(v)$ is
a finite set for every $v\in E^0$, then the graph is called {\it row-finite}. In this note we will consider only row-finite graphs. In this setting, if the number of vertices, i.e., $|E^0|$, is finite, then the number of edges, i.e., $|E^1|$, is finite as well
and we call $E$ a {\it finite graph}.
\end{remark}

\begin{definition} [{\sc Weighted Leavitt path algebra}]\label{def3}
Let $(E,\omega)$ denote a weighted graph and $R$ a unital ring. Set $X:=E^0\cup E^1\cup(E^1)^*$, where 
$(E^1)^*=\{\alpha_i^*\ | \ \alpha_i\in E^1\}$. The quotient $R\X/I$ of the free $R$-ring $R\X$ generated by $X$ and the ideal $I$ of $R\X$ generated by the relations
\begin{enumerate}[(1)]
\item $vw = \delta_{vw}v$ for every $v,w \in E^0$,
\medskip 

\item $s(\alpha)\alpha_i = \alpha_i r(\alpha) = \alpha_i$ and $r(\alpha)\alpha_i^*= \alpha_i^*s(\alpha) = \alpha_i$ for all $\alpha\in E^{\st}$ and $1\leq i \leq \omega(\alpha)$,

\medskip 

\item $\sum\limits_{\{\alpha\in E^{\st}, s(\alpha)=v\}}\alpha_i\alpha_j^*= \delta_{ij}v$ for all $v\in E^0$ and $1\leq i, j\leq \max\{\omega(\alpha) \ |  \ \alpha\in E^{\st}, s(\alpha) = v\}$,
\medskip

\item $\sum\limits_{1\leq i\leq \max\{\omega(\alpha),\omega(\beta)\}}\alpha_i^*\beta_i= \delta_{\alpha\beta}r(\alpha)$, for all $\alpha,\beta\in E^{\st}$
\end{enumerate}
is called {\it weighted Leavitt path algebra of $(E,\omega)$} and is denoted by $L_R(E,\omega)$ or just $L(E,\omega)$. In relations (3) and (4), we set $\alpha_i$ and $\alpha_i^*$ zero whenever $i > \omega(\alpha)$. When $R$ is not commutative, then we consider $L_R(E,\omega)$ as a left $R$-module.
\end{definition}

Weighted Leavitt path algebras are involutary graded rings with unit if $E^0$ is finite and local units otherwise. In fact,  the weighted Leavitt path algebra $L_R(E,\omega)$ is a $\mathbb Z^n$-graded ring, where $n=\max\{\omega(\alpha) \mid \alpha \in E^{\st}\}$. 
The grading defined as follows:  for $v\in E^0$ define $\deg(v)=0$ and  for $\alpha \in E^{\st}$,  $\deg(\alpha_i)=e_i$ and $\deg(\alpha_i^*)=-e_i$, $1\leq i\leq \omega(\alpha)$,  where $\alpha_i \in E^1$ (note that the grading depends on the ordering of edges  $E^1$).  Here $e_i$ denotes the element of $\mathbb Z^n$ whose $i-th$ component is $1$ and whose other components are $0$ 

\begin{example}\label{wlpapp}
Let $K$ be a field. Then the weighted Leavitt path algebra of a weighted graph consisting of one vertex and $n+k$ loops of weight $n$ 
is $L_K(n,n+k)$. To show this, let $E^{\st}=\{y_1,\dots,y_{n+k}\}$ with $\omega(y_i)=n$, $1\leq i\leq n+k$. Denote the $n$ edges corresponding to the structure edge $y_i\in E^{\st}$ by $\{y_{1i},\dots,y_{ni}\}$. We visualise this data as follows:
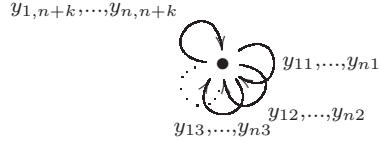
\begin{equation*}
\xymatrix{
& \bullet \ar@{.}@(l,d) \ar@(ur,dr)^{y_{11},\dots,y_{n1}} \ar@(r,d)^{y_{12},\dots,y_{n2}} \ar@(dr,dl)^{y_{13},\dots,y_{n3}} \ar@(l,u)^{y_{1,n+k},\dots,y_{n,n+k}}& 
}
\end{equation*}
Set $x_{sr}=y_{rs}^*$ for $1\leq r\leq n$ and $1\leq s\leq n+k$ and arrange the $y$'s and $x$'s in the matrices
\begin{equation*} 
Y=\left( 
\begin{matrix} 
y_{11} & y_{12} & \dots & y_{1,n+k}\\ 
y_{21} & y_{22} & \dots & y_{2,n+k}\\ 
\vdots & \vdots & \ddots & \vdots\\ 
y_{n1} & y_{n2} & \dots & y_{n,n+k} 
\end{matrix} 
\right), \qquad 
X=\left( 
\begin{matrix} 
x_{11\phantom{n+{},}} & x_{12\phantom{n+{},}} & \dots & x_{1n\phantom{n+{},}}\\ 
x_{21\phantom{n+{},}} & x_{22\phantom{n+{},}} & \dots & x_{2n\phantom{n+{},}}\\ 
\vdots & \vdots & \ddots & \vdots\\ 
x_{n+k,1} & x_{n+k,2} & \dots & x_{n+k,n} 
\end{matrix} 
\right) 
\end{equation*} 
Then condition~(3) of Definition~\ref{def1} precisely says that $Y\cdot X=I_{n,n}$ and condition~(4) is equivalent to 
$X\cdot Y=I_{n+k,n+k}$ which are the generators of $L_K(n,n+k)$. 
\end{example}

\begin{example}\label{exex1}
Let $(E,\omega)$ be a weighted graph where $w: E^{\st}\rightarrow \N$ is the constant map $\omega(\alpha)=1$ for all $\alpha \in E^{\st}$. Then $E^{\st}=E^1$ and $L(E,\omega)$ is isomorphic to the usual Leavitt path algebra $L(E)$. 
\end{example}

\begin{example}
In Example~\ref{wlpapp}, the map defined by $y_{1i}\mapsto y_i$, $x_{i1}\mapsto {y_i}^*$, $1\leq i \leq k+1$, $y_{i,i+k}\mapsto 1$, $x_{i+k,i}\mapsto 1$ $2\leq i \leq n$ and $y_{ij}\mapsto 0$, $x_{ji}\mapsto 0$ otherwise, induces a surjective ring homomorphism 
\[L_K(n,n+k) \longrightarrow L_K(1,k+1)\] showing the Leavitt algebra $L(1,k+1)$ is a quotient of $L(n,n+k)$. In Theorem~\ref{thmn1} we show that a simple weighted Leavitt path algebras has to be a simple Leavitt path algebra. 
\end{example}

\begin{example}
Consider a weighted graph with one vertex and one structured edge $\alpha$ of weight $n$, i.e., $E^1=\{\alpha_1,\dots,\alpha_n\}$  and an unweighted graph $F$ with one vertex and $n$ edges $\{\alpha_1,\dots,\alpha_n\}$. Then the map $(E,\omega)\longrightarrow F, \alpha_i \mapsto \alpha_i^*,$
induces an isomorphism on the level of LPAs, namely
\[\xymatrix{
L\big(\!\!\!\!\!\!\!\!\!\!\!\!\!& \bullet\ar@(ul,ur)^{\alpha_{1},\dots,\alpha_{n}}
},\omega\big)\cong
\xymatrix{
L\big(\!\!\!&   \bullet \ar@{.}@(l,d) \ar@(ur,dr)^{\alpha_{1}} \ar@(r,d)^{\alpha_{2}} \ar@(dr,dl)^{\alpha_{3}} 
\ar@(l,u)^{\alpha_{n}}& 
}\big).\]
Note that this isomorphism is not graded as $L(E,\omega)$ is $\mathbb Z^n$-graded, whereas $L(F)$ is just $\mathbb Z$-graded. The graph $F$ is called the unweighted graph associated with $E$ (see Definition~\ref{defn3}). 
\end{example}

Until the end of this section $R$ denotes a unital ring and $(E,\omega)$ a weighted graph. For any $v\in E^0$ which is not a sink  (i.e.,  $s^{-1}(v)\neq\emptyset$) fix an $\alpha^{v}\in E^{\st}$ such that 
\begin{equation}\label{specialedge}
s(\alpha^{v})=v\text{ and } \omega(\alpha^{v})=\omega(v)
\end{equation}
where $\omega(v)=\max\{\omega(\alpha) \  |  \  \alpha\in E^{\st},s(\alpha)=v\}$.

\begin{definition}[{\sc Generalised path}]\label{deffgp}
Set $s(v):=v$, $r(v):=v$, $s(\alpha_i):=s(\alpha)$, $r(\alpha_i):=r(\alpha)$, $s(\alpha^*_i):=r(\alpha)$ and $r(\alpha^*_i):=s(\alpha)$ for any $v\in E^0$, $\alpha \in E^{\st}$ and $1\leq i \leq \omega(\alpha)$. Let $\X$ denote the set of all nonempty words over $X:=E^0\cup E^1\cup(E^1)^*$. A word $p\in \X$ is called {\it a generalised path} if either $p=x_1x_2\dots x_n$ for some $n\geq 1$ and $x_1,\dots,x_n\in E^1\cup (E^1)^*$ such that $r(x_i)=s(x_{i+1})$, $1\leq i \leq n-1$ or $p=x_1$ for some $x_1\in E^0$. The {\it length} $|p|$ of a generalised path $p=x_1\dots x_n$ is $n$ if $n\geq 1$ and $x_1,\dots,x_n\in E^1\cup (E^1)^*$ or $0$ if $n=1$ and $x_1\in E^0$. $p$ is called {\it trivial} if $|p|=0$ and {\it nontrivial} if $|p|\geq 1$. Further we set $s(p):=s(x_1)$, and $r(p):=r(x_n)$.  
\end{definition}

\begin{definition}[{\sc Normal element of $R\X$}]\label{deffne}
A word $A\in\X$ is called {\it word of type I} if $A=\alpha^v_i(\alpha^v_j)^*$ for some $v\in E^0$ which is not a sink and some $1\leq i,j\leq \omega(\alpha^v)$. $A$ is called {\it word of type II} if $A=\alpha^*_1\beta_1$ for some $\alpha,\beta\in E^{\st}$. A generalised path is called {\it normal} if it does not contain any subwords of type I or type II. An element of $R\X$ is called {\it normal} if it lies in the linear span $R\X_N$ of all normal generalised paths.
\end{definition}

We will show that any element of $L_R(E,\omega)$ has precisely one normal representative in $R\X$. For this we need some definitions and results from \cite{bergman78} which we will recall below. Note that a weighted Leavitt path algebra is a quotient of a free $R$-ring where $R$ is a not necessarily commutative ring while in \cite{bergman78} free associative, unital algebras over commutative rings are considered. Hence we have to make a few adaptations (see Remark~\ref{hhmmnn}).

Here we recall the basics of Bergman's diamond machinery needed in the paper. Until the end of the proof of Theorem~\ref{thm10}, $R$ denotes a unital ring and $X$ any set. By an {\it $R$-ring} we mean a (not necessarily unital) ring which is an $R$-bimodule such that the multiplication is left linear in the first argument and right linear in the second one. By an {\it ideal of an $R$-ring $A$} we mean an ideal of the ring $A$ which is an $R$-subbimodule of $A$. Let $\langle X\rangle$ denote the semigroup (with juxtaposition) of all nonemtpy words over $X$ and set $\overline{\X}:=\X\cup\{\text{empty word}\}$. Further let $R\X$ denote the {\it free $R$-ring generated by $X$}, i.e. the free left $R$-module generated by $\X$ made an $R$-ring by the multiplication $(\sum\limits_{x\in\X}r_xx)(\sum\limits_{y\in\X}s_yy)=\sum\limits_{x,y\in\X}r_xs_yxy$.  

\begin{definition}[{\sc Reduction system}]\label{def5}
 Let $S$ be a set of pairs of the form $\sigma= (W_\sigma,f_\sigma)$, where $W_\sigma\in \langle X\rangle$ and $f_\sigma\in R\langle X\rangle$ such that all coefficients of $f_\sigma$ lie in the center of $R$. Then $S$ is called a {\it reduction system} for $R\X$.
For any $\sigma\in S$ and $A, B\in \overline{\X}$, let $r_{A\sigma B}$ denote the $R$-bimodule endomorphism of $R\X$ that maps $AW_\sigma B$ to $Af_\sigma B$ and fixes all other elements of $\X$. The maps $r_{A\sigma B}: R\X\rightarrow R\X$ are called {\it reductions}.

\end{definition}

Until the end of the proof of Theorem~\ref{thm10}, $S$ denotes a reduction system for $R\X$.

\begin{definition} [{\sc Irreducible element, final sequence of reduction}]\label{def6}
We shall say a reduction $r_{A\sigma B}$ acts {\it trivially} on an element $a\in R\X$ if the coefficient of $AW_\sigma B$ in $a$ is zero, and we shall call $a$ {\it irreducible (under S)} if every reduction is trivial on $a$, i.e., if $a$ involves none of the monomials $AW_\sigma B$. The $R$-subbimodule of all irreducible elements of $R\X$ will be denoted $R\X_{\irr}$. A finite sequence of reductions $r_1,\dots, r_n$ will be said to be {\it final} on $a\in R\X$ if $r_n\dots r_1(a)\in R\X_{\irr}$.

\end{definition}

\begin{definition}[{\sc Reduction-finite element, reduction-unique element}]\label{def7}
An element $a$ of $R\X$ will be called {\it reduction-finite} if for every infinite sequence $r_1, r_2 ,\dots$ of reductions, $r_i$ acts trivially on $r_{i-1} \dots r_1(a)$ for all sufficiently large $i$. If $a$ is reduction-finite, then any maximal sequence of reductions $r_i$, such that each $r_i$ acts nontrivially on $r_{i-l}\dots r_1(a)$, will be finite, and hence a final sequence. It follows from their definition that the reduction-finite elements form an $R$-subbimodule of $R\X$. We shall call an element $a\in R\X$ {\it reduction-unique} if it is reduction-finite, and if its images under all final sequences of reductions are the same. This common value will be denoted $r_S(a)$. The set of reduction-unique elements  of $R\X$ forms an $R$-subbimodule, and $r_S$ is a bilinear map (i.e. an $R$-bimodule homomorphism) of this submodule into $R\X_{\irr}$ (see \cite[proof of Lemma 1.1(i)]{bergman78}).
\end{definition}

\begin{definition}[{\sc Ambiguity, resolvable ambiguity}]\label{def8}
A $5$-tuple $(\sigma, \tau, A, B, C)$ with $\sigma,\tau\in S$ and $A, B, C\in \X$,
such that $W_\sigma= AB$ and $W_\tau = BC$ is called an {\it overlap ambiguity} of $S$. We shall say the overlap ambiguity $(\sigma, \tau, A, B, C)$ is {\it resolvable} if there exist compositions of reductions, $r$ and $r'$, such that $r(f_\sigma C) = r'(Af_\tau)$. Similarly, a $5$-tuple $(\sigma, \tau, A, B, C)$ with $\sigma\neq\tau$ and $A,B,C\in \overline{\X}$ will be called an {\it inclusion ambiguity} if $W_\sigma = B$, $W_\tau = ABC$ and such an ambiguity will be called {\it resolvable} if $Af_\sigma C$ and $f_\tau$ can be reduced to a common expression.
\end{definition}

\begin{definition} [{\sc Semigroup partial ordering compatible with $S$}]\label{def9}
By a {\it semigroup partial ordering} on $\X$ we shall mean a partial order $\leq$ such
that 
\[B < B'\Rightarrow ABC < AB'C\]
for any $B, B'\in \X,A,C\in\overline{\X}$. We call $\leq$ {\it compatible} with $S$ if for all $\sigma\in S$, $f_\sigma$ is a linear combination of monomials $< W_\sigma$.

\end{definition}

We are in a position to state Bergman's diamond lemma~\cite[Theorem~1.2]{bergman78}. This theorem will be used to find a basis for the weighted Leavitt path algebras.

\begin{theorem} \label{thm10}
Let $\leq$ be a semigroup partial ordering on $\X$ compatible
with $S$ and having descending chain condition. Then the following conditions are
equivalent:
\begin{enumerate}[\upshape(1)]
\item All ambiguities of $S$ are resolvable.
\medskip

\item All elements of $R\X$ are reduction-unique under $S$.

\medskip 

\item $R\X_{\irr}$ is a set of representatives for the elements of the $R$-ring $R\X/I$, where $I$ is the ideal of $R\X$ generated by the elements $W_\sigma-f_\sigma~(\sigma\in S)$.
\end{enumerate}
When these conditions hold, $R\X/I$ may be identified with the $R$-bimodule $R\X_{\irr}$ made an $R$-ring by the multiplication $a\cdot b = r_S(ab)$.
\end{theorem}

Now we can use the previous theorem in order to prove that any element of $L_R(E,\omega)$ has precisely one normal representative in $R\X$ where $X=E^0\cup E^1\cup (E^1)^*$.

\begin{theorem} \label{thm11}
Let $R$ be a unital ring and $(E,\omega)$ a row-finite weighted graph. Then the weighted Leavitt path algebra $L_R(E,\omega)$ has a basis consisting of normal generalised paths. Namely, the basis elements are of the form $p=x_1\dots x_n$, $x_i \in E^1 \cup (E^1)^*$, $r(x_i)=s(x_{i+1})$, $1\leq i \leq n-1$ or $p=x_1$, $x_1\in E^0$ such that none of the words $\alpha^v_i(\alpha^v_j)^*$ where $v\in E^0$ is not a sink and $1\leq i,j\leq \omega(\alpha^v)$ and $\alpha^*_1\beta_1$ where $\alpha,\beta\in E^{\st}$ is a subword of $p$. 

\end{theorem}

\begin{proof} 
In order to be able to apply Theorem~\ref{thm10}, we replace the relations (1)-(4) in Definition~\ref{def3} by the relations
\begin{enumerate}[(1')]
\item For any $v,w \in E^0$, \[vw = \delta_{vw}v,\]
\medskip

\item For any $v \in E^0$, $\alpha\in E^{\st}$ and $1\leq i \leq \omega(\alpha)$,
\begin{align*}
v\alpha_i&=\delta_{vs(\alpha)}\alpha_i,\\
\alpha_iv&=\delta_{vr(\alpha)}\alpha_i,\\
v\alpha_i^*&=\delta_{vr(\alpha)}\alpha_i^* \text{ and}\\
\alpha_i^*v&=\delta_{vs(\alpha)}\alpha_i^*
\end{align*}

\medskip

\item For any $\alpha,\beta\in E^{\st}$, $1\leq i \leq \omega(\alpha)$ and $1\leq j \leq \omega(\beta)$,

\begin{align*}
\alpha_i\beta_j&=0\text{ if  } r(\alpha)\neq s(\beta),\\
\alpha_i^*\beta_j&=0\text{ if  } s(\alpha)\neq s(\beta),\\
\alpha_i\beta_j^*&=0\text{ if  } r(\alpha)\neq r(\beta)\text{ and}\\
\alpha_i^*\beta_j^*&=0\text{ if  } s(\alpha)\neq r(\beta)\\
\end{align*}

\medskip

\item For all $v\in E^0$ which are not sinks and $1\leq i, j\leq \omega(\alpha^v)$, 
\[\alpha^v_i(\alpha^v_j)^*= \delta_{ij}v-\sum\limits_{\substack{\alpha\in E^{\st}, s(\alpha)=v\\\alpha\neq \alpha^v}}\alpha_i\alpha_j^*\]
and

\medskip 

\item For all $\alpha,\beta\in E^{\st}$ such that $s(\alpha)=s(\beta)$, 
\[\alpha_1^*\beta_1= \delta_{\alpha\beta}r(\alpha)-\sum\limits_{2\leq i\leq \max\{\omega(\alpha),\omega(\beta)\}}\alpha_i^*\beta_i.\]

\end{enumerate}
In relations (4') and (5'), we set $\alpha_i$ and $\alpha_i^*$ zero whenever $i > \omega(\alpha)$. Clearly the relations (1')-(5') generate the same ideal $I$ of $R\X$ as the relations (1)-(4). Denote by $S$ the reduction system for $R\X$ defined by the relations (1')-(5') (i.e.,  $S$ is the set of all pairs $\sigma=(W_\sigma,f_\sigma)$ where $W_\sigma$ equals the left hand side of an equation in (1')-(5') and $f_\sigma$ the corresponding right hand side).\\
For any $A=x_1\dots x_n\in \X$ set $l(A):=n$ and $m(A):= \big |\{i\in\{1,\dots,n-1\}|x_ix_{i+1} \text{ is of type I or II}\} \big |$. Define a partial ordering $\leq$ on $\X$ by 
\[A\leq B\Leftrightarrow \big [A=B\big ]~\lor~\big[l(A)<l(B)\big]~\lor \big[l(A)=l(B)~\land~ \forall C,D\in \overline\X:m(CAD)<m(CBD)\big].\]
Clearly $\leq$ is a semigroup partial ordering on $\X$ compatible with $S$ and the descending chain condition is satisfied.\\
It remains to show that all ambiguities of $S$ are resolvable. In the table below we list all types of ambiguities which may occur.
\begin{table}[!htbp]
\resizebox{17.7cm}{!}{
\begin{tabular}{|c||c|c|c|c|c|}
\hline &(1')&(2')&(3')&(4')&(5')\\
\hhline{|=#=|=|=|=|=|} (1')&$uvw$&$vw\alpha_i$,$vw\alpha_i^*$ &-&-&-\\
\hline (2')&$\alpha_ivw$, $\alpha_i^*vw$&$v\alpha_iw$, $v\alpha_i^*w$, $\alpha_iv\beta_j$,  etc.&$v\alpha_i\beta_j$, $v\alpha_i\beta_j^*$ etc.&$w\alpha_i^v(\alpha_j^v)^*$&$v\alpha_1^*\beta_1$\\
\hline (3')&-&$\alpha_i\beta_jv$, $\alpha_i^*\beta_jv$ etc.&$\alpha_i\beta_j\gamma_k$, $\alpha_i\beta_j\gamma_k^*$ etc.&$\beta_k\alpha_i^v(\alpha_j^v)^*$, $\beta_k^*\alpha_i^v(\alpha_j^v)^*$&$\gamma_k\alpha_1^*\beta_1$, $\gamma_k^*\alpha_1^*\beta_1$\\
\hline (4')&-&$\alpha_i^v(\alpha_j^v)^*w$&$\alpha_i^v(\alpha_j^v)^*\gamma_k$, $\alpha_i^v(\alpha_j^v)^*\gamma_k^*$&-&$\alpha_i^v(\alpha_1^v)^*\beta_1$\\
\hline (5')&-&$\alpha_1^*\beta_1v$&$\alpha_1^*\beta_1\gamma_i$, $\alpha_1^*\beta_1\gamma_i^*$&$\alpha_1^*\alpha_1^v(\alpha_j^v)^*$&-\\
\hline
\end{tabular}
}
\end{table}\\
Note that there are no inclusion ambiguities. The ((4')-(5') and (5')-(4')) ambiguities $\alpha_i^v(\alpha_1^v)^*\beta_1$ and $\alpha_1^*\alpha_1^v(\alpha_j^v)^*$, where $v\in E^0$ is not a sink, $1\leq i,j\leq \omega(\alpha^v)$ and $\alpha,\beta\in E^{\st}$ such that $s(\alpha)=s(\beta)=v$ are the ones which are most difficult to resolve. We will show how to resolve the ambiguity $\alpha_i^v(\alpha_1^v)^*\beta_1$ and leave the other cases to the reader.
\xymatrixcolsep{-7pc}
\xymatrixrowsep{4pc}
\[\xymatrix{
&\alpha^v_i(\alpha_1^v)^*\beta_1\ar[rd]^-{(5')}\ar[ld]_-{(4')}&
\\
{\begin{array}{cc}&(\delta_{i1}v-\sum\limits_{\substack{s(\alpha)=v\\\alpha\neq \alpha^v}}\alpha_i\alpha_1^*)\beta_1\\=&\delta_{i1}v\beta_1-\sum\limits_{\substack{s(\alpha)=v\\\alpha\neq \alpha^v}}\alpha_i\alpha_1^*\beta_1\end{array}}\ar[d]_-{(5')}
&&
{\begin{array}{cc}&\alpha^v_i(\delta_{\alpha^v\beta}r(\alpha^v)-\sum\limits_{j=2}^{\omega(\alpha^v)}(\alpha^v_j)^*\beta_j)\\=&\delta_{\alpha^v\beta}\alpha^v_ir(\alpha^v)-\sum\limits_{j=2}^{\omega(\alpha^v)}\alpha^v_i(\alpha^v_j)^*\beta_j\end{array}\ar[d]^-{(4')}}
\\
{\begin{array}{cc}&\delta_{i1}v\beta_1\\&-\sum\limits_{\substack{s(\alpha)=v\\\alpha\neq \alpha^v}}\alpha_i(\delta_{\alpha\beta}r(\alpha)-\sum\limits_{j=2}^{\omega(\alpha^v)}\alpha_j^*\beta_j)\\=&\delta_{i1}v\beta_1-\sum\limits_{\substack{s(\alpha)=v\\\alpha\neq \alpha^v}}\delta_{\alpha\beta}\alpha_ir(\alpha)\\&+\sum\limits_{\substack{s(\alpha)=v\\\alpha\neq \alpha^v}}\sum\limits_{j=2}^{\omega(\alpha^v)}\alpha_i\alpha_j^*\beta_j\\=&\delta_{i1}v\beta_1-\delta_{\beta\neq\alpha^v}\beta_ir(\beta)\\&+\sum\limits_{\substack{s(\alpha)=v\\\alpha\neq \alpha^v}}\sum\limits_{j=2}^{\omega(\alpha^v)}\alpha_i\alpha_j^*\beta_j\end{array}}\ar[rd]_-{(2')}
&&
{\begin{array}{cc}&\delta_{\alpha^v\beta}\alpha^v_ir(\alpha^v)\\&-\sum\limits_{j=2}^{\omega(\alpha^v)}(\delta_{ij}v-\sum\limits_{\substack{s(\alpha)=v\\\alpha\neq \alpha^v}}\alpha_i\alpha_j^*)\beta_j\\=&\delta_{\alpha^v\beta}\alpha^v_ir(\alpha^v)-\sum\limits_{j=2}^{\omega(\alpha^v)}\delta_{ij}v\beta_j\\&+\sum\limits_{j=2}^{\omega(\alpha^v)}\sum\limits_{\substack{s(\alpha)=v\\\alpha\neq \alpha^v}}\alpha_i\alpha_j^*\beta_j\\=&\delta_{\alpha^v\beta}\beta_ir(\beta)-\delta_{i\geq 2}v\beta_i\\&+\sum\limits_{j=2}^{\omega(\alpha^v)}\sum\limits_{\substack{s(\alpha)=v\\\alpha\neq \alpha^v}}\alpha_i\alpha_j^*\beta_j\end{array}}\ar[ld]^-{(2')}
\\
&{\begin{array}{cc}&\delta_{i1}\beta_1-\delta_{\beta\neq\alpha^v}\beta_i+\sum\limits_{\substack{s(\alpha)=v\\\alpha\neq \alpha^v}}\sum\limits_{j=2}^{\omega(\alpha^v)}\alpha_i\alpha_j^*\beta_j\\=&\delta_{\alpha^v\beta}\beta_i-\delta_{i\geq 2}\beta_i+\sum\limits_{j=2}^{\omega(\alpha^v)}\sum\limits_{\substack{s(\alpha)=v\\\alpha\neq \alpha^v}}\alpha_i\alpha_j^*\beta_j.\end{array}}&}
\]
It follows from Theorem~\ref{thm10}, that $R\X_{\irr}$ is a set of representatives for the elements of $R\X/I=L_R(E,\omega)$. But clearly $R\X_{\irr}=R\X_N$.
\end{proof}

In \cite{zel12} a basis for a Leavitt path algebras were described. Here we obtain this result as a corollary of Theorem~\ref{thm11}. 

\begin{corollary}
Let $E$ be a directed graph and $L_R(E)$ the associated Leavitt path algebra. Then the monomials $pq^*$, where $p=x_1\dots x_n$ and $q=y_1\dots y_m$ are paths (of possibly length zero) such that $x_ny_m^*\neq\alpha^v (\alpha^v)^*$ for any $v\in E^0$ which is not a sink, form a basis for $L_R(E)$. 
\end{corollary}

In Section \ref{sec3} we will use Theorem~\ref{thm11} to determine when a weighted Leavitt path algebra is simple resp. graded simple. In Section \ref{sec7} we will use it to determine when a weighted Leavitt path algebra is a domain. In order to do this we need the concept of normal forms.

\begin{definition}[{\sc Normal form of an element of $L_R(E,\omega)$}]\label{def12}
Let $a\in L_R(E,\omega)$. Then the unique normal representative of $a\in R\X$ is called the {\it normal form} of $a$ and is denoted by $\NF(a)$. It follows from~\cite[Lemma 1.1]{bergman78} that
\begin{align*}
\NF: L_R(E,\omega)&\longrightarrow R\X_N\\
a&\longmapsto \NF(a)
\end{align*} 
is an isomorphism of $R$-bimodules (note that $\NF=r_S$). If we make $R\X_N$ an $R$-ring by defining $\NF(a)\cdot \NF(b):=\NF(ab)$, then $\NF$ is an isomorphism of $R$-rings.
\end{definition}

\begin{remark}\label{hhmmnn}
As mentioned above, for the diamond lemma, Bergman's starting point is a unital free algebra. One can state and use the diamond lemma in the setting of non-unital free algebras as our treatment in this section. However one can also start with $X=E^0\cup E^1 \cup {E^1}^*$ and consider the unital free algebra on $X$ subject to the weighted Leavitt path algebra relations. This gives the \emph{unitisation} ring $L_R(E,\omega)\times R$. One can then conduct the proof of Theorem~\ref{thm11} in this ring. It is then easy to obtain the normal forms for $L_R(E,\omega)$ from this setting as well. 
\end{remark}

\section{Classification of simple and graded simple weighted Leavitt path algebras}\label{sec3}
In this section $R$ denotes a ring and $(E,\omega)$ a weighted graph. As usual, we call an ideal $J$ of $L_R(E,\omega)$ {\it proper} if $J\neq\{0\}$ and $J\neq L_R(E,\omega)$. Note that an ideal of the ring $L_R(E,\omega)$ is the same as an ideal of the $R$-ring $L_R(E,\omega)$ since $L_R(E,\omega)$ has local units.

In Definition \ref{defn2} we define reducible and irreducible weighted graphs. We will show that if $(E,\omega)$ is reducible, then $L_R(E,\omega)$ is isomorphic to $L_R(F)$ for some unweighted graph $F$. It is an open question if there are examples of irreducible graphs $(E,\omega)$ such that $L_R(E,\omega)$ is isomorphic to $L_R(F)$ for some unweighted graph $F$. However we will show that if $(E,\omega)$ is irreducible, then $L_R(E,\omega)$ is not graded simple. The main idea to show that $L_R(E,\omega)$ is not graded simple provided $(E,\omega)$ is irreducible is to find a nontrivial lr-normal generalised path (cf. Definition \ref{deflrn}). These are normal generalised paths $p$ which have the property that if $o$ and $q$ are nontrivial normal generalised paths such that $r(o)=s(p)$ and $s(q)=r(p)$, then $opq$ again is a normal generalised path. It is easy to see that a nontrivial lr-normal generalised path generates a proper graded ideal (follows from the uniqueness of the normal form).

\begin{definition}[{\sc Path, tree, cycle}]\label{defn0}
A generalised path $x_1\dots x_n$ is called a {\it path} if $n=1$ and $x_1\in E^0$ or $n\geq 1$ and $x_1,\dots,x_n\in E^1$. If $u,v\in E^0$ and there is a path $p$ such that $s(p)=u$ and $r(p)=v$, then we write $u\geq v$. Clearly $\geq$ is a preorder on $E^0$. If $u\in E^0$ then $T(u):=\{v\in E^0 \ |  \ u\geq v\}$ is called {\it tree of $u$}. A nontrivial path $p$ such that $v = s(p) = r(p)$ is called a {\it closed path based at $v$}. If $p = x_1\dots x_n$ is a closed path based at $v = s(p)$ and $s(x_i)\neq s(x_j)$ for every
$i\neq j$, then $p$ is called a {\it cycle}. 
\end{definition}

\begin{definition}[{\sc Connected components, dual of a generalised path}]\label{defn00}
If $u,v\in E^0$ and there is a generalised path $p$ such that $s(p)=u$ and $r(p)=v$, then we write $u\geq_g v$. Clearly $\geq_g$ is an equivalence relation on $E^0$. The equivalence classes of $\geq_g$ are called {\it connected components}. $(E,\omega)$ is called {\it connected} if there is only one connected component. Set $v^*:=v$ for any $v\in E^0$ and $(\alpha_i^*)^*:=\alpha_i$ for any $\alpha_i\in E^1$. If $p=x_1\dots x_n$ is a generalised path, then $p^*:=x^*_n\dots x_1^*$ is called {\it dual of $p$}. Note that $p^*$ is a generalised path such that $s(p^*)=r(p)$ and $r(p^*)=s(p)$.
\end{definition}

\begin{definition}[{\sc Circle graph, line graph, oriented line graph}]\label{defn1}
A weighted graph $(E,\omega)$ is called {\it cyclic} if it contains a cycle and {\it acyclic} otherwise. $(E,\omega)$ is called a {\it circle graph}, if it is connected, cyclic and $|s^{-1}(v)|,|r^{-1}(v)|\leq 1$ for any $v\in E^0$. $(E,\omega)$ is called a {\it line graph} if it is connected, acyclic and $|s^{-1}(v)|+|r^{-1}(v)|\leq 2$ for any $v\in E^0$. $(E,\omega)$ is called an {\it oriented line graph} if it is a line graph such that $|s^{-1}(v)|,|r^{-1}(v)|\leq 1$ for any $v\in E^0$.
\end{definition}

\begin{definition}[{\sc Unweighted graph, weight forest, reducible graph, irreducible graph}]\label{defn2}
A weighted graph $(E,\omega)$ is called {\it unweighted} if $\omega= 1$. If $(E,\omega)$ is unweighted, we identify $E^1$ and $E^{\st}$ and write $L_R(E)$ instead of $L_R(E,\omega)$ (see Example~\ref{exex1}). An $\alpha\in E^{\st}$ is called {\it weighted} if $\omega(\alpha)>1$ and {\it unweighted} otherwise. The set of all weighted elements of $E^{\st}$ is denoted by $E^{\st}_\omega$. A $v\in E^0$ is called {\it weighted} if $\omega(v)>1$ and {\it unweighted} otherwise. The set of all weighted elements of $E^0$ is denoted by $E^{0}_\omega$. The set $\overline{E^{0}_\omega}:=\bigcup\limits_{v\in E^{0}_\omega}T(v)$ is called {\it weight forest of $(E,\omega)$}. A weighted graph 
$(E,\omega)$ with $E^{0}_\omega\neq \emptyset$  is called {\it reducible} if $|s^{-1}(v)|,|r^{-1}(v)\cap s^{-1}(\overline{E^{0}_\omega})|\leq 1$  for any 
$v\in \overline{E^{0}_\omega}$ and {\it irreducible} otherwise. 
\end{definition}

Consider the weighted graph $(E',\omega')$ one gets by dropping all vertices which do not belong to the weight forest $\overline{E^{0}_\omega}$ and all structured edges $\alpha$ such that $s(\alpha)$ or $r(\alpha)$ does not belong to the weight forest. One checks easily that $(E,\omega)$ is reducible if and only if all connected components of $(E',\omega')$ are either circle graphs or oriented line graphs.

\begin{example}\label{exred}
Let $(E,\omega)$ be the weighted graph below. 
\[
\xymatrix{&x\ar[d]^{\gamma}&y\ar[d]^{\delta}\\
u  \ar[r]_{\alpha_1,\alpha_2}  & v\ar[r]_{\beta} &w 
}
\]
Note that $\overline{E^{0}_\omega}=\{u,v,w\}$. Then 
\[
\xymatrix{u  \ar[r]_{\alpha_1,\alpha_2}  & v\ar[r]_{\beta} &w 
}
\]
is the weighted graph $(E',\omega')$ one gets as described in the paragraph after Definition \ref{defn2}. Since the only connected component of $(E',\omega')$ is an oriented line graph, $(E,\omega)$ is reducible.
\end{example}

\begin{example}\label{exirred}
Let $(E,\omega)$ be the weighted graph below. 
\[
\xymatrix{&x\ar[d]^{\gamma}&y\ar[d]^{\delta_1,\delta_2}\\
u  \ar[r]_{\alpha_1,\alpha_2}  & v\ar[r]_{\beta} &w 
}
\]
Note that $\overline{E^{0}_\omega}=\{u,v,w,y\}$. Then 
\[
\xymatrix{&&y\ar[d]^{\delta_1,\delta_2}\\
u  \ar[r]_{\alpha_1,\alpha_2}  & v\ar[r]_{\beta} &w 
}
\]
is the weighted graph $(E',\omega')$ one gets as described in the paragraph after Definition \ref{defn2}. Since the only connected component of $(E',\omega')$ is neither a circle graph nor an oriented line graph, $(E,\omega)$ is irreducible.
\end{example}

\begin{definition}[{\sc The unweighted graph associated with $(E,\omega)$}]\label{defn3}
Let $(E,\omega)$ be a weighted graph. We construct a unweighted graph $F$ as follows. Let $F^0=E^0$, $F^1=\{e_{\alpha_i}|\alpha_i\in E^1\}$, $s(e_{\alpha_i})=s(\alpha)$ and $r(e_{\alpha_i})=r(\alpha)$ if $s(\alpha)\not\in\overline{E^{0}_\omega}$ and $s(e_{\alpha_i})=r(\alpha)$ and $r(e_{\alpha_i})=s(\alpha)$ if $s(\alpha)\in\overline{E^{0}_\omega}$. The graph $F$ is called {\it the unweighted graph associated with $E$.}
\end{definition}

Thus  $F$ has the same vertices as $E$, an edge $\alpha_i\in E^1$ is kept as it is 
if $s(\alpha)\not\in\overline{E^{0}_\omega}$ and it is reversed if $s(\alpha)\in\overline{E^{0}_\omega}$. 

\begin{example}
The unweighted $F$ associated with the reducible weighted graph $(E,\omega)$ from Example \ref{exred} is the graph 
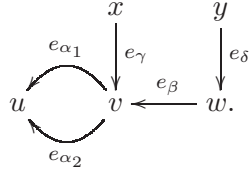
\[
\xymatrix{&x\ar[d]^{e_\gamma}&y\ar[d]^{e_\delta}\\
u    & v\ar@/^1.2pc/[l]^{e_{\alpha_2}}\ar@/_1.2pc/[l]_{e_{\alpha_1}} &w. \ar[l]_{e_\beta}
}
\]
It follows from the next proposition that $L_R(E,\omega)\cong L_R(F)$.
\end{example}

\begin{proposition} \label{propn1}
If $(E,\omega)$ is reducible, then $L_R(E,\omega)\cong L_R(F)$, where $F$ is the unweighted graph associated with $(E,\omega)$.
\end{proposition}

\begin{proof} 
Let $X:=E^0\cup E^1\cup (E^1)^*$ and $X':=F^0\cup F^1\cup (F^1)^*$. Then the bijection $X\rightarrow X'$ mapping $v\mapsto v$ for any $v\in E^0$, $\alpha_i\mapsto e_{\alpha_i}$ and $\alpha_i^*\mapsto e^*_{\alpha_i}$ for any $\alpha_i\in E^1$ such that $s(\alpha)\not\in\overline{E^{0}_\omega}$ and $\alpha_i\mapsto e_{\alpha_i}^*$ and $\alpha_i^*\mapsto e_{\alpha_i}$ for any $\alpha_i\in E^1$ such that $s(\alpha)\in\overline{E^{0}_\omega}$ induces an isomorphism $f:R\X\rightarrow R\langle X'\rangle$. Let $I$ and $I'$ be ideals of $R\X$  and $R\langle X'\rangle$ generated by the relations (1)-(4) in Definition~\ref{def3}, respectively 
 (hence $L_R(E,\omega)=R\X/I$ and $L_R(F)=R\langle X'\rangle/I'$, see Example~\ref{exex1}). In order to show that $L_R(E,\omega)\cong L_R(F)$ it suffices to show that $f(I)=I'$. Set
\[A^{(1)}:=\big \{vw-\delta_{vw}v \ |  \ v,w\in E^0 \big \},\]
\[A^{(2)}:=\big \{s(\alpha)\alpha_i-\alpha_i, \alpha_ir(\alpha)-\alpha_i, r(\alpha)\alpha_i^*-\alpha_i^*, \alpha_i^*s(\alpha)-\alpha^*_i \ | \ \alpha\in E^{\st}, 1\leq i\leq \omega(\alpha)\big \},\]
and for any $v\in E^0$ which is not a sink 
\[A^{(3)}_v:=\Big\{\sum\limits_{\alpha\in s^{-1}(v)}\alpha_i\alpha_j^*-\delta_{ij}v \ |  \ 1\leq i,j\leq \omega(v)\Big\}\]
and
\[A^{(4)}_v:=\Big\{\sum\limits_{i=1}^{\max\{\omega(\alpha),\omega(\beta)\}}\alpha_i^{*}\beta_i-\delta_{\alpha\beta}r(\alpha)\ | \ \alpha,\beta\in s^{-1}(v)\Big\}.\] 
Then $I$ is generated by $A^{(1)}$, $A^{(2)}$, the $A^{(3)}_v$'s and the $A^{(4)}_v$'s. Define $B^{(1)},B^{(2)},B^{(3)}_v,B^{(4)}_v$ $\in R\langle X'\rangle$, where $v\in F$ is not a sink, analogously. Then $I'$ is generated by $B^{(1)}$, $B^{(2)}$, the $B^{(3)}_v$'s and the $B^{(4)}_v$'s. Clearly $f(A^{(1)})=B^{(1)}$ and $f(A^{(2)})=B^{(2)}$. Let $v\in E^0$ be not a sink. One checks easily that if $v\in E^0\setminus \overline{E^{0}_\omega}$, then $f(A_v^{(3)})=B_v^{(3)}$ and $f(A_v^{(4)})=B_v^{(4)}$. Now assume that $v\in\overline{E^{0}_\omega}$. Then $s^{-1}(v)=\{\alpha\}$ for some $\alpha\in E^{\st}$ since $v$ is not a sink and $(E,\omega)$ is reducible. Set $w:=r(\alpha)$. Then
\[A^{(3)}_v=\big \{\alpha_i\alpha_j^*-\delta_{ij}v  \  | \  1\leq i,j\leq \omega(\alpha) \big \}\]
and
\[A^{(4)}_v=\Big\{\sum\limits_{i=1}^{\omega(\alpha)}\alpha_i^{*}\alpha_i-w\Big\}.\] 
It is easy to show that $s^{-1}(w)=\{e_{\alpha_1},\dots,e_{\alpha_{\omega(\alpha)}}\}$ in $F$ (note that all edges which $w$ emits in $E$ get reversed since clearly $w\in \overline{E^{0}_\omega}$; further $r^{-1}(w)\cap s^{-1}(\overline{E^{0}_\omega})=\{\alpha\}$ in $E$ since $(E,\omega)$ is reducible). Hence 
\[B^{(3)}_w=\Big\{\sum\limits_{i=1}^{\omega(\alpha)}e_{\alpha_i}e_{\alpha_i}^*-w\Big\}\]
and
\[B^{(4)}_w=\big  \{e_{\alpha_i^{*}}e_{\alpha_j}-\delta_{ij}v\  | \  1\leq i,j\leq \omega(\alpha) \big  \}.\]
Clearly $f(A^{(3)}_v)=B^{(4)}_w$ and $f(A^{(4)}_w)=B^{(3)}_w$. It follows that $f(I)=I'$ (note that for any $w\in \overline{E^{0}_\omega}$ which is not a sink in $F$, there is a $v\in \overline{E^{0}_\omega}$ and an $\alpha\in E^{\st}$ such that $s(\alpha)=v$ and $r(\alpha)=w$). Thus $L_R(E,\omega)\cong L_R(F)$.
\end{proof}

%\textnormal{In the following we will show that if $(E,\omega)$ is irreducible, then it is neither simple nor graded simple.}
One idea to show that $L_R(E,\omega)$ is not graded simple provided $(E,\omega)$ is irreducible, is to find nontrivial lr-normal generalised paths. We define them below. 

\begin{definition}[{\sc l-normal, r-normal, lr-normal generalised paths}]\label{deflrn}
An element $x\in X=E^0\cup E^1\cup (E^1)^*$ is called {\it l-normal}, iff there is no $y\in X$ such that $yx$ is of type I or II (see Definition \ref{deffne}). $x$ is called {\it r-normal}, iff there is no $y\in X$ such that $xy$ is of type I or II. $x$ is called {\it lr-normal}, iff it is l-normal and r-normal. More generally a normal gen. path $p=x_1...x_n$ is called {\it l-normal} if $x_1$ is l-normal, {\it r-normal} if $x_n$ is r-normal and {\it lr-normal} if it is l-normal and r-normal.
\end{definition}

Let $p$ be a nontrivial lr-normal gen. path and $J$ the ideal of $L_R(E,\omega)$ generated by (the image of) $p$. One checks easily that $\NF(J)$ is the linear span of the normal generalised paths containing $p$ as a subword (note that if $o$ and $q$ are nontrivial normal generalised paths such that $r(o)=s(p)$ and $s(q)=r(p)$, then $opq$ again is a normal generalised path). It follows from the uniqueness of the normal form (see Theorem~\ref{thm11}) that $J\neq L_R(E,\omega)$ (for instance $J$ does not contain any vertex). Hence nontrivial lr-normal gen. paths generate proper graded ideals.

\begin{lemma}\label{lemn1}
If there is a $v\in E^{0}$ such that $s^{-1}(v)$ contains two distinct weighted structured edges $\alpha,\beta$, then there is a nontrivial lr-normal gen. path. Hence $L_R(E,\omega)$ is not graded simple.
\end{lemma}

\begin{proof}
The normal form defined in Section~\ref{sec2} depends on the choice of elements $\alpha^v\in s^{-1}(v)~(v\text{ not a sink})$ such that $\omega(\alpha^v)$ is maximal (see~(\ref{specialedge})). Without loss of generality assume that $\omega(\alpha)\geq \omega(\beta)$. Then clearly one can choose $\alpha^v\neq\beta$. One checks easily that $\beta_2$ is lr-normal.
\end{proof}
An example of an irreducible weighted graph satisfying the condition of the previous lemma is the following weighted graph:
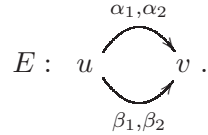
\[
\xymatrix{E:\!\!\!\!\!\!\!\!&\!\!\!\! u  \ar@/^1.2pc/[r]^{\alpha_1,\alpha_2}\ar@/_1.2pc/[r]_{\beta_1,\beta_2}  & v  
}.
\]
Unfortunately there are irreducible weighted graphs without nontrivial lr-normal gen. paths, for example
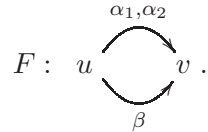
\[
\xymatrix{F:\!\!\!\!\!\!\!\!&\!\!\!\! u  \ar@/^1.2pc/[r]^{\alpha_1,\alpha_2}\ar@/_1.2pc/[r]_{\beta}  & v  
}.
\] 
One can use the Diamond Lemma to show that the ideal of  $L_R(F,\omega)$ generated by $\alpha_1$ is proper, see the proof of the next lemma.
\begin{lemma}\label{lemn2}
If there is a $v\in E^{0}$ such that $s^{-1}(v)$ contains a weighted structured edge $\alpha$ and an unweighted structured edge $\beta$, then $L_R(E,\omega)$ is not graded simple.
\end{lemma}

\begin{proof}
By the previous lemma we can assume that $\alpha$ is the only weighted structured edge in $s^{-1}(v)$. Hence we can choose $\alpha^v=\alpha$. Let $J$ be the ideal of $L_R(E,\omega)=R\X/I$ generated by $\alpha_1+I$. Using Bergman's machinery (see~\S\ref{sec2}) we will show that $L_R(E,\omega)/J$ is not trivial which implies that $J\neq L_R(E,\omega)$.\\
It is easy to show that $L_R(E,\omega)/J$ is isomorphic to the quotient $R\X/I'$ where $I'$ is the ideal of $R\X$ generated by the relations (1)-(4) in Definition~\ref{def3} and the relation 
\begin{enumerate}[(5)]
\item $\alpha_1=0$.
\end{enumerate} 
We call the words $\alpha_1, \alpha_1^*, \beta_1\beta_1^*,\alpha_2^*\alpha_2\in \X$ {\it words of type III}. Further we call a generalised path {\it strongly normal} if it is normal and does not contain a subword of type III. An element of $R\X$ is called {\it strongly normal} if it lies in the linear span $R\X_{\SN}$ of all strongly normal generalised paths. Using Theorem~\ref{thm10} we will show that $R\X_{\SN}$ is a set of representatives for the elements of  $R\X/I'$.\\
In order to be able to apply Theorem~\ref{thm10} we replace the relations (1)-(5) by the relations (1')-(5') in the proof of 
Theorem~\ref{thm11} and the relations 
\begin{enumerate}[(1')]
\setcounter{enumi}{5}
\item $\alpha_1=0$, $\alpha_1^*=0$,

\medskip

\item $\beta_1\beta_1^*=v-\sum\limits_{\substack{\gamma\in s^{-1}(v)\\\gamma\neq\alpha,\beta}}\gamma_1\gamma_1^*$ and

\medskip

\item $\alpha_2^*\alpha_2=r(\alpha)-\sum\limits_{i=3}^{\omega(\alpha)}\alpha_i^*\alpha_i$.
\end{enumerate}
Clearly the relations (1')-(8') generate the same ideal $J$ of $R\X$ as the relations (1)-(5). Denote by $S$ the reduction system for $R\X$ defined by the relations (1')-(8') (i.e., $S$ is the set of all pairs $\sigma=(W_\sigma,f_\sigma)$ where $W_\sigma$ equals the left hand side of an equation in (1')-(8') and $f_\sigma$ the corresponding right hand side). For any $A=x_1\dots x_n\in \X$ set $l(A):=n$, 
\[m_{I,II}(A):= |\{i\in\{1,\dots,n-1\}\  |\  x_ix_{i+1} \text{ is of type I or II}\}|\]
and
\begin{align*}m_{III}(A):= |\{i\in\{1,\dots,n\} \ |\ &\text{either }x_i \text{ is of type III}\\&\text{or } i\leq n-1 \text{ and }x_ix_{i+1} \text{ is of type III}\}|.
\end{align*}
Define a partial ordering $\leq$ on $\X$ by 
\begin{align*}
&A\leq B\\
\Leftrightarrow~& [A=B]~\lor~[l(A)<l(B)]~\lor\\&[l(A)=l(B)~\land~ \forall C,D\in \overline\X:m_{I,II}(CAD)<m_{I,II}(CBD)]~\lor\\& [l(A)=l(B)~\land~ \forall C,D\in \overline\X:m_{I,II}(CAD)\leq m_{I,II}(CBD)~\land\\
&\forall C,D\in \overline\X:m_{III}(CAD)<m_{III}(CBD)].
\end{align*} 
Clearly $\leq$ is a semigroup partial ordering on $\X$ compatible with $S$ and the descending chain condition is satisfied. Further it is easy to show that all ambiguities of $S$ are resolvable. For example
\xymatrixcolsep{3pc}
\xymatrixrowsep{3pc}
\[\xymatrix{
&\alpha_2^*\alpha_2\alpha_2^*\ar[rd]^-{(8')}\ar[ld]_-{(4')}&
\\
\alpha_2^*v\ar[rd]_-{(2')}
&&
\hspace{-1.3cm}\big(r(\alpha)-\sum\limits_{i=3}^{\omega(\alpha)}\alpha_i^*\alpha_i\big)\alpha_2^*\ar[ld]^-{(2'),(4')}
\\
&\alpha_2^*.&}
\]
It follows from Theorem~\ref{thm10}, that $R\X_{\irr}$ is a set of representatives for the elements of $R\X/I'$. But clearly $R\X_{\irr}=R\X_{\SN}$.
It follows that $R\X/I'$ has more than one element since $R\X_{\SN}$ has more than one element (for example $v,\alpha_2,\alpha_2^*,\beta_1,\beta_1^*\in R\X_{\SN}$). Since $L_R(E,\omega)/J\cong R\X/I'$, it follows that $L_R(E,\omega)/J$ has more than one element and hence $J\neq L_R(E,\omega)$. Thus $J$ is a proper graded ideal.
\end{proof}

We can now use Lemma~\ref{lemn1} and Lemma~\ref{lemn2} to prove that if $(E,\omega)$ is irreducible, then $L_R(E,\omega)$ is not graded simple (note that there are irreducible weighted graphs which neither satisfy the condition of Lemma~\ref{lemn1} nor the condition of Lemma~\ref{lemn2}, e.g. the weighted graph $(E,\omega)$ from Example \ref{exirred}).
\begin{proposition}\label{propn2}
If $(E,\omega)$ is irreducible, then $L_R(E,\omega)$ is not graded simple.
\end{proposition}

\begin{proof} 
Since $(E,\omega)$ is irreducible, there is a $v\in \overline{E^{0}_\omega}$ such that $|s^{-1}(v)|>1$ or $|r^{-1}(v)\cap s^{-1}(\overline{E^{0}_\omega})|>1$.\\

\begin{enumerate}

\item[Case 1] {\it Assume that $|s^{-1}(v)|>1$.}\\
Since $v\in\overline{E^{0}_\omega}$, there is a $u\in E^{0}_\omega$ and a path $p$ such that $s(p)=u$ and $r(p)=v$. 
By Lemma~\ref{lemn1} and Lemma~\ref{lemn2}, we may assume that $s^{-1}(u)=\{\alpha\}$ for some $\alpha\in E^{\st}_\omega$. It follows that $p$ is nontrivial and $p=\alpha_ip'$ for some $1\leq i\leq \omega(\alpha)$ and a path $p'$ (here we allow $p'$ to be the empty word). Choose a $\beta\in s^{-1}(v)$ such that $\beta\neq \beta^v$ (this is possible as $|s^{-1}(v)|>1$). One checks easily that $\alpha_2p'\beta_1$ is lr-normal.

\medskip 

\item[Case 2] {\it Assume $|r^{-1}(v)\cap s^{-1}(\overline{E^{0}_\omega})|>1$.}\\
Since $|r^{-1}(v)\cap s^{-1}(\overline{E^{0}_\omega})|>1$, there are $u_1,u_2\in \overline{E^{0}_\omega}$ and distinct $\alpha,\beta\in E^{\st}$ such that $s(\alpha)=u_1$, $s(\beta)=u_2$ and $r(\alpha)=r(\beta)=v$. Since $u_1,u_2\in\overline{E^{0}_\omega}$, there are $w_1, w_2\in E^{0}_\omega$ and paths $p_1$ and $p_2$ such that $s(p_1)=w_1$, $r(p_1)=u_1$, $s(p_2)=w_2$ and $r(p_2)=u_2$. By Lemma~\ref{lemn1} and Lemma~\ref{lemn2}, we may assume that $s^{-1}(w_1)=\{\gamma\}$ and $s^{-1}(w_1)=\{\epsilon\}$ for some $\gamma, \epsilon\in E^{\st}_\omega$. Assume that $p_1$ and $p_2$ are nontrivial. Then $p_1=\gamma_ip_1'$ and $p_2=\epsilon_jp_2'$ for some $1\leq i\leq \omega(\gamma)$, $1\leq j\leq \omega(\epsilon)$ and paths $p'_1$ and $p'_1$ (here we allow $p'_1$ and $p'_2$ to be the empty word). One checks easily that $\gamma_2p'_1\alpha_1\beta_1^*(p'_2)^*\epsilon_2^*$ is lr-normal. The case that $p_1$ or $p_2$ is trivial can be handled analogously. \qedhere
\end{enumerate}
\end{proof}

\begin{example}
Let $(E,\omega)$ be the irreducible weighted graph from Example \ref{exirred}. One checks easily that $\alpha_2\beta_1\delta_2^*$ is lr-normal. Hence $L_R(E,\omega)$ is not graded simple.
\end{example}

We are ready to classify the simple weighted Leavitt path algebras. 
%The following theorem shows that although weighted Leavitt path algebras provide a much large classe of rings, such as $L_R(n,n+k)$, than Leavitt path algebras, but if $L(E,\omega)$ is ($\mathbb Z^n$-graded) simple, then $L(E,\omega)\cong L(F)$, where $F$ is the unweighted graph associated to $E$ and $L(F)$ is ($\mathbb Z$-graded) simple. 

\begin{theorem}[{\sc Simplicity Theorem}]\label{thmn1}
The weighted Leavitt path algebra $L_R(E,\omega)$ is simple if and only if $(E,\omega)$ is reducible and $L_R(F)$ is simple where $F$ is the unweighted graph associated with $(E,\omega)$.
\end{theorem}

\begin{proof}
Follows from Proposition~\ref{propn1} and Proposition~\ref{propn2}.
\end{proof}

Theorem~\ref{thmn1} shows that although weighted Leavitt path algebras produce a wide range of algebras which are not covered by Leavitt path algebras (such as $L(n,n+k)$, $n\geq 2$), the class of simple weighted algebras doesn't produce new examples. 

We can prove a graded version of Theorem \ref{thmn1} in the case that $R$ is a field. Note that $L_R(E,\omega)$ is $\mathbb Z^n$-graded where $n=\max\{\omega(\alpha)~|~\alpha\in E^{st}\}$ while $L_R(F)$ is $\mathbb Z$-graded.

\begin{theorem}[{\sc Graded Simplicity Theorem}]\label{thmn2}
If $R$ is a field, then the weighted Leavitt path algebra $L_R(E,\omega)$ is graded simple if and only if $(E,\omega)$ is reducible and $L_R(F)$ is graded simple where $F$ is the unweighted graph associated with $(E,\omega)$.
\end{theorem}

\begin{proof}

($\Rightarrow$) Assume that $L_R(E,\omega)$ is graded simple. Then $(E,\omega)$ is reducible by Proposition~\ref{propn2} and hence $L_R(E,\omega)\cong L_R(F)$ by Proposition~\ref{propn1}. Assume that $L_R(F)$ contains a proper graded ideal $J$. Then $J$ is generated by elements of $F^0$ by \cite[Theorem 2.5.8]{abrams-ara-molina} (namely $J$ is generated by a hereditary and saturated subset of $F^0$). But the isomorphism between $L_R(E,\omega)$ and $L_R(F)$ established in Proposition~\ref{propn1} maps $E^0$ onto $F^0$. Therefore the image of $J$ in $L_R(E,\omega)$ is generated by elements of $E^0$ and therefore it is a proper graded ideal of $L_R(E,\omega)$. But this contradicts the assumption that $L_R(E,\omega)$ is graded simple. Thus $L_R(F)$ is graded simple.

($\Leftarrow$)
Assume that $L_R(F)$ is graded simple and $(E,\omega)$ is reducible. Let $\phi:L_R(E,\omega)\rightarrow L_R(F)$ be the isomorphism induced by the map $f:R\X\rightarrow R\langle X'\rangle$ defined in Proposition~\ref{propn1}. The only hereditary and saturated subsets of $F^0$ are $\emptyset$ and $F^0$ by  \cite[Theorem 2.5.8]{abrams-ara-molina}. It follows from 
\cite[Theorem 2.8.10]{abrams-ara-molina}  that every proper ideal $J$ of $L_R(F)$ is generated by terms of the form $v+r_1c+\dots+r_mc^m$ where $m\geq 1$, $r_1,\dots,r_m\in R$, $r_m\neq 0$ and $c$ is a cycle based at $v$ without exit. But the elements $f^{-1}(v), f^{-1}(r_1c),\dots,f^{-1}(r_mc^m)$ are homogeneous in $L_R(E,\omega)$. Assume that $f^{-1}(J)$ is a graded ideal of $L_R(E,\omega)$, then, by the definition of a graded ideal, all the elements $f^{-1}(v), f^{-1}(r_1c),\dots,f^{-1}(r_mc^m)$ are contained in $f^{-1}(J)$ and hence all the elements $v, r_1c,\dots,r_mc^m$ are contained in $J$. But this implies that $J$ is graded which is a contradiction. Thus $L_R(E,\omega)$ is graded simple.
\end{proof}

\section{LV-algebras and classification of weighted Leavitt path algebras which are domains}\label{sec7}

In \cite{vitt56} Leavitt proved that the algebra $L_\mathbb{Z}(2,3)$ is a domain. Namely he defined normal forms for the elements of $L_\mathbb{Z}(2,3)$ and showed that the map $\nu$ which associates to each $x\in L_\mathbb{Z}(2,3)$ the degree of $\NF(x)$ as a polynomial in the generators of $L_\mathbb{Z}(2,3)$ (with the convention $\nu(0)=-\infty$) is a valuation, i.e. $\nu(xy)=\nu(x)+\nu(y)$. It follows that if $0=xy$, then $-\infty=\nu(0)=\nu(xy)=\nu(x)+\nu(y)$ and hence $x$ or $y$ must be $0$. Later Cohn \cite{cohn66} proved, using the same method, that the Leavitt algebras $L_K(n,n+k)$, $n\geq 2$ are domains if $K$ is a field.

Here we adapt Leavitt's approach to study certain weighted Leavitt path algebras. Clearly there is no valuation on $L_R(E,\omega)$ if there is more than one vertex (since in this case there are zero divisors). But for a large class of weighted Leavitt path algebras, so-called LV-algebras (Definition~\ref{lvalgebra3}),  one can define a ``local'' valuation. Using the local valuation we prove that LV-algebras are prime, semiprimitive and non-singular, similar to the case of Leavitt path algebras. However they are not (graded) von Neumann regular. Thus we obtain a much larger class of prime and nonsingular rings than Leavitt path algebras. 

\begin{definition}[\sc Support of an element of $L_R(E,\omega)$]
Let $(E,\omega)$ be a weighted graph and $R$ a ring. If $a\in L_R(E,\omega)$, then the set $\supp(a)$ of all normal generalised paths occurring in $\NF(a)$ with nonzero coefficient is called the {\it support of $a$}.
\end{definition}
\begin{definition}[{\sc local valuation}]\label{def2}
Let $(E,\omega)$ be a weighted graph and $R$ a ring. A {\it local valuation on $L_R(E,\omega)$} is a map $\nu:L_R(E,\omega)\longrightarrow \N_0\cup\{-\infty\}$ such that 
\begin{enumerate}[(1)]
\item \label{1}$\nu(a)=-\infty$ if and only if $a=0$,
\smallskip
\item \label{2}$\nu(a)=0$ if and only if $a\neq 0$ and $\supp(a)\subseteq E^0$,
\smallskip

\item \label{3}$\nu(a+b)\leq \max\{\nu(a),\nu(b)\}$ for any $a,b\in L_R(E,\omega)$ and

\smallskip

\item \label{4}$\nu(ab)=\nu(a)+\nu(b)$ for any $v\in E^0$, $a\in L_R(E,\omega)v$ and $b\in vL_R(E,\omega)$.
\end{enumerate}
We use the conventions $-\infty\leq x$ and $x+(-\infty)=(-\infty)+x=-\infty$ for any $x\in\N_0\cup\{-\infty\}$.
\end{definition}

For a certain type of weighted Leavitt path algebras, we can construct local valuations. Let $(E,\omega)$ be a weighted graph and set $\nu:=\deg\circ \NF$ (for a more formal definition of $\nu$ see Proposition \ref{thm4}). Assume that $E^{\st}$ contains an unweighted structured edge $\alpha$.  Then $\nu(\alpha_1^*\alpha_1)=\nu(r(\alpha))=0\neq 2=\nu(\alpha_1^*)+\nu(\alpha_1)$ by relation (4) in Definition \ref{def3}. Hence $\nu$ is not a local valuation. Assume now that there is a $v\in E^0$ and an $\alpha\in s^{-1}(v)$ such that $\omega(\alpha)>\omega(\beta)$ for any $\beta\in s^{-1}(v)\setminus\{ \alpha\}$. Then $\nu(\alpha_{\omega(\alpha)}^*\alpha_{\omega(\alpha)})=\nu(v)=0\neq 2=\nu(\alpha_{\omega(\alpha)}^*)+\nu(\alpha_{\omega(\alpha)})$ by relation (3) in Definition \ref{def3} and again $\nu$ is not a local valuation. This motivates the following definition.

\begin{definition}[{\sc LV-graph, LV-rose, LV-algebra}]\label{lvalgebra3}
A weighted graph $(E,\omega)$ is called an {\it LV-graph} if the condition
\[\omega(\alpha)\geq 2~\forall\alpha\in E^{\st} \text{ and }  |   \{\alpha\in s^{-1}(v)\ | \ \omega(\alpha)=\omega(v)\}|\geq 2~\forall v\in E^0,v\text{ not a sink}\tag{LV}\label{LV}\]
is satisfied. Recall that $\omega(v)=\max\{\omega(\alpha)~|~\alpha\in s^{-1}(v)\}$ for any $v\in E^0$ which is not a sink. In order to simplify the exposition, we additionally require that a LV-graph has edges (i.e.,  $E^{\st}\neq\emptyset$) and is connected (see Definition \ref{defn00}). An LV-graph $(E,\omega)$ such that $|E^0|=1$ is called an {\it LV-rose}. The weighted Leavitt path algebras $L_R(E,\omega)$ where $R$ is a ring and $(E,\omega)$ is an LV-graph are called {\it LV-algebras}.
\end{definition}
%\remark{If $\nu$ is a local valuation, then 
%\[\nu(pq)=\begin{cases}
%-\infty,\text{ if }r(p)\neq s(q),\\
%\nu(p)+\nu(q),\text{ if }r(p)=s(q),
%\end{cases}\]
%for any normal generalised paths $p$ and $q$.}

\begin{example}\label{exlvrose}
The weighted graph  
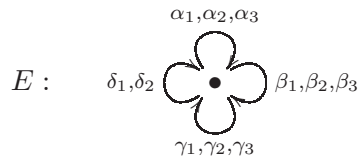
\begin{equation*}
\xymatrix{
E:\hspace{1cm}& \bullet  \ar@(ur,dr)^{\beta_1,\beta_2,\beta_3} \ar@(dr,dl)^{\gamma_1,\gamma_2,\gamma_3} \ar@(dl,ul)^{\delta_1,\delta_2}\ar@(ul,ur)^{\alpha_1,\alpha_2,\alpha_3}& 
}
\end{equation*}
is an LV-rose. We will see later that if $K$ is a field, then $L_K(E,\omega)$ is a domain which is not
isomorphic to any of the algebras $L_K(n,n+k)$.
\end{example}

For a generalised path $p$, recall the length $|p|$ from Definition~\ref{deffgp}.

\begin{proposition}\label{thm4}
If $(E,\omega)$ is an LV-graph and $R$ a domain, then the map
\begin{align*}
\nu:L_R(E,\omega)&\longrightarrow \N_0\cup\{-\infty\}\\
a&\longmapsto \max\{|p|\ | \ p\in \supp(a)\}.
\end{align*}
is a local valuation on $L_R(E,\omega)$. Here we use the convention $\max(\emptyset)=-\infty$.
\end{proposition}
\begin{proof}
Obviously (\ref{1}), (\ref{2}) and (\ref{3}) hold. It remains to show (\ref{4}). Let $v\in E^0$, $a\in L_R(E,\omega)v$ and $b\in vL_R(E,\omega)$. If one of the terms $\nu(a)$ and $\nu(b)$ equals $0$ or $-\infty$, then clearly  $\nu(ab)=\nu(a)+\nu(b)$. Suppose now $\nu(a),\nu(b)\geq 1$. Clearly $\nu(ab)\leq \nu(a)+\nu(b)$ since a reduction preserves or decreases the length of a generalised path. It remains to show that 
$\nu(ab)\geq \nu(a)+\nu(b)$. Let 
\[p_k=x^k_1\dots x^k_{\nu(a)}~(1\leq k \leq r)\]
be the elements of $\supp(a)$ with maximal length (namely $\nu(a)$) and 
\[q_l=y^l_1\dots y^l_{\nu(b)}~(1\leq l \leq s)\]
be the elements of $\supp(b)$ with maximal length (namely $\nu(b)$). We assume that the $p_k$'s are pairwise distinct and also that the $q_l$'s are pairwise distinct. Since $\NF$ is a linear map, we have 
\begin{displaymath}
\NF(p_k q_l)= \left\{
\begin{array}{ll}
p_k q_l & \text{if }x^k_{\nu(a)}y^l_1\text{ is not of type I or II},\\
\\
\NF([\delta_{ij}x^k_1\dots x^k_{\nu(a)-1}y^l_2\dots y^l_{\nu(b)}])\\-\sum\limits_{\substack{\alpha\in s^{-1}(u),\\\alpha\neq \alpha^u}}x^k_1\dots x^k_{\nu(a)-1}\alpha_i\alpha_j^*y^l_2\dots y^l_{\nu(b)} & \text{if }x^k_{\nu(a)}y^l_1\text{ is of type I,}\\
\\
\NF([\delta_{\alpha\beta}x^k_1\dots x^k_{\nu(a)-1}y^l_2\dots y^l_{\nu(b)}])\\-\sum\limits_{2\leq i\leq \max\{\omega(\alpha),\omega(\beta)\}}x^k_1\dots x^k_{\nu(a)-1}\alpha_i^*\beta_iy^l_2\dots y^l_{\nu(b)} & \text{if }x^k_{\nu(a)}y^l_1\text{ is of type II.}
\end{array}\right.
\end{displaymath}

\medskip

\begin{enumerate}

\item[Case 1]   {\it Assume that $x^k_{\nu(a)}y^l_1$ is not of type I or II for any $k,l$}.\\
Then $p_k q_l\in \supp(ab)$ for any $k,l$. It follows that $\nu(ab)\geq |p_k q_l|=\nu(a)+\nu(b)$.
 
\medskip

\item[Case 2] {\it Assume that there are $k,l$ such that $x^k_{\nu(a)}y^l_1$ is of type I.}\\
Then there are a $u\in E^0$ and $1\leq i,j \leq \omega(\alpha^u)$ such that $x^k_{\nu(a)}y^l_1=\alpha^u_i(\alpha_j^u)^*$. Choose an $\alpha\neq\alpha^u$ of weight $\omega(u)$. This is possible since $(E,\omega)$ is an LV-graph.

\medskip

\begin{enumerate}

\item [Case 2.1] {\it Assume $p_{k'} q_{l'}\neq x^k_1\dots x^k_{\nu(a)-1}\alpha_i\alpha_j^*y^l_2\dots y^l_{\nu(b)}$ for any $k',l'$.}\\
Then
\[x^k_1\dots x^k_{\nu(a)-1}\alpha_i\alpha_j^*y^l_2\dots y^l_{\nu(b)}\in \supp(ab)\]
since it doesn't cancel with another term. It follows that $\nu(ab)\geq \nu(a)+\nu(b)$. 

\medskip 

\item[Case 2.2]  {\it Assume $p_{k'}  q_{l'}= x^k_1\dots x^k_{\nu(a)-1}\alpha_i\alpha_j^*y^l_2\dots y^l_{\nu(b)}$ for some $k',l'$.}\\
One checks easily that in this case 
\[p_k  q_{l'}=x^k_1\dots x^k_{\nu(a)-1}\alpha^u_i\alpha_j^*y^l_2\dots y^l_{\nu(b)}\in \supp(ab).\]
It follows that $\nu(ab)\geq \nu(a)+\nu(b)$.

\end{enumerate}

\medskip 
\item[Case 3] {\it Assume that there are $k,l$ such that $x^k_{\nu(a)}y^l_1$ is of type II.}\\
Then there are $\alpha,\beta\in E^{\st}$ such that $x^k_{\nu(a)}y^l_1=\alpha_1^*\beta_1$. Since $(E,\omega)$ is an LV-graph, $\omega(\alpha),\omega(\beta)\geq 2$.

\medskip

\begin{enumerate}

\item [Case 3.1] {\it Assume $p_{k'} q_{l'}\neq x^k_1\dots x^k_{\nu(a)-1}\alpha_2^*\beta_2y^l_2\dots y^l_{\nu(b)}$ for any $k',l'$. }\\
Then
\[x^k_1\dots x^k_{\nu(a)-1}\alpha_2^*\beta_2y^l_2\dots y^l_{\nu(b)}\in \supp(ab)\]
since it doesn't cancel with another term. It follows that $\nu(ab)\geq \nu(a)+\nu(b)$. 

\medskip 

\item[Case 3.2]  {\it Assume $p_{k'}  q_{l'}= x^k_1\dots x^k_{\nu(a)-1}\alpha_2^*\beta_2y^l_2\dots y^l_{\nu(b)}$ for some $k',l'$.}\\
One checks easily that in this case 
\[p_k  q_{l'}=x^k_1\dots x^k_{\nu(a)-1}\alpha_1^*\beta_2y^l_2\dots y^l_{\nu(b)}\in \supp(ab)\]
It follows that $\nu(ab)\geq \nu(a)+\nu(b)$.

\end{enumerate}

\end{enumerate}

Hence (\ref{4}) also holds and thus $\nu$ is a local valuation on $L_R(E,\omega)$.
\end{proof}
\begin{theorem}\label{cor1}
Let $(E,\omega)$ be a weighted graph and $R$ a ring. Then $L_R(E,\omega)$ is a domain if and only if $R$ is a domain and $(E,\omega)$ is either an unweighted rose with not more than one petal or an LV-rose.
\end{theorem}
\begin{proof}
One checks easily that $L_R(E,\omega)$ is the zero ring or has zero divisors if $R$ is not a domain or $(E,\omega)$ is neither an unweighted rose with not more than one petal nor an LV-rose. Suppose now that $R$ is a domain. If $(E,\omega)$ is a rose with no petals, then $L_R(E,\omega)\simeq R$ and if $(E,\omega)$ is an unweighted rose with one petal, then $L_R(E,\omega)\simeq R[X,X^{-1}]$. Hence $L_R(E,\omega)$ is a domain in these cases. If $(E,\omega)$ is an LV-rose, then there is a local valuation on $L_R(E,\omega)$ by the previous proposition. It follows from (\ref{1}) and (\ref{4}) in Definition \ref{def2} that $L_R(E,\omega)$ is a domain.
\end{proof}

We recover the theorem of Leavitt~\cite[footnote~6]{vitt62} and Cohn~\cite{cohn66}. 

\begin{corollary}
For a domain $R$, the Leavitt algebras $L_R(n,n+k)$, $n\geq 2$, are domains. 
\end{corollary}

In contrast to the fact that the class of weighted Leavitt path algebras does not contain any new examples of simple algebras, it contains new examples of domains. We use the dependence number, a ring-invariant introduced by Cohn in \cite{cohn66}, in order to prove that there are weighted Leavitt path algebras which are domains but are not isomorphic to any of Leavitt's algebras.

\begin{definition}[{\sc Filtration, valuation}]
Let $R$ be a ring. A {\it (positive increasing) filtration on $R$} is a map $\nu:R\rightarrow \N_0\cup\{-\infty\}$ such that 
\begin{equation}\label{eqn2}
\nu(x)=-\infty\Leftrightarrow x=0,\quad\nu(x-y)\leq \max\{\nu(x),\nu(y)\},\quad\nu(xy)\leq \nu(x)+\nu(y)\quad\forall x,y\in R.
\end{equation}
Let $\nu$ be a filtration on $R$ and set $R_n:=\{x\in R\mid \nu(x)\leq n\}$ for any $n\in \N_0\cup\{-\infty\}$. Then each $R_n$ is an additive subgroup of $R$ and 
\begin{equation}\label{eqn3}
R_mR_n\subseteq R_{m+n}~\forall m,n\in \N_0,\quad \bigcup\limits_{n\in \N_0}R_n=R,\quad \{0\}=R_{-\infty}\subseteq R_0\subseteq R_1\subseteq\dots.
\end{equation}
Conversely if $\{R_n\mid n\in \N_0\}$ is a family of additive subgroups of $R$ such that (\ref{eqn3}) holds, then the map $\nu:R\rightarrow \N_0\cup\{-\infty\}$ defined by $\nu(x)=\min\{n\in\N_0\cup\{-\infty\}\mid x\in R_n\}$, is a filtration on $R$. Hence fixing a filtration on $R$ is the same as fixing a family of additive subgroups of $R$ such that (\ref{eqn3}) holds. Every ring has the {\it trivial filtration} $\nu$ defined by $\nu(0)=-\infty$ and $\nu(x)=0~\forall x\neq 0$. A filtration $\nu$ on $R$ such that $\nu(xy)=\nu(x)+\nu(y)~\forall x,y\in R$ is called a {\it valuation}.
\end{definition}

\begin{definition}[{\sc Dependence number of a ring}]
Let $R$ be a ring and $\nu$ a filtration on $R$.
\begin{enumerate}[(1)]
\item A subset $X$ of $R$ is called {\it $R$-dependent} if $X = \{0\}$ or if $X = \{x_1, \dots , x_r\}$ and there exist $a_1,\dots,a_r\in R$ such that
\[\nu(x_1) + \nu(a_1) = \dots= \nu(x_r) + \nu(a_r) > v(\sum\limits_{i=1}^r x_ia_i)\].
\item An element $y\in R$ is called {\it $R$-dependent on a subset $X$} of $R$ if $y = 0$ or if there exist $x_1,\dots,x_r\in X$ and $a_1, \dots , a_r\in R$ such that
\[\nu(y - \sum\limits_{i=1}^rx_ia_i)< \nu(y),\hspace{0.5cm}  \nu(x_i)+\nu(a_i)\leq\nu(y)\quad (i = 1, \dots, r).\]
\end{enumerate}
Further, a subset $X$ of $R$ is called {\it strongly $R$-dependent} if it is $R$-dependent
and any element of maximal value in $X$ is $R$-dependent on the remaining elements of $X$.
The {\it dependence number of $R$ relative to $\nu$}, $\lambda_\nu(R)$, is the least integer $n$ for which there exists an $R$-dependent set of $n$ elements which is not strongly $R$-dependent. The supremum of the $\lambda_\nu(R)$ for all filtrations
$\nu$ on $R$ is called the {\it dependence number of $R$} and is denoted by $\lambda(R)$.
\end{definition}

\begin{theorem}
Let $K$ be a field and $(E,\omega)$ an LV-rose such that the minimal weight is $2$. Then the dependence number of $L_K(E,\omega)$ equals $2$.
\end{theorem}
\begin{proof}
Set $L:=L_K(E,\omega)$. Let $\alpha\in E^{\st}$ be of weight $2$ and choose a $\beta\in E^{\st}$ such that $\beta\neq\alpha$ (possible since any LV-algebra contains at least $2$ structured edges). Let $\nu$ be valuation on $L$ defined in Proposition \ref{thm4}. By relation (4) in Definition \ref{def3} we have $\alpha_1^*\beta_1+\alpha_2^*\beta_2=0$. It follows that the set $\{\alpha_1^*,\alpha_2^*\}$ is $L$-dependent with respect to $\nu$. But $\{\alpha_1^*,\alpha_2^*\}$ is not strongly $L$-dependent since $\nu(\alpha_1^*-\alpha_2^*x)\geq \nu(\alpha_1^*)$ for any $x\in L$. Hence $\lambda_\nu(L)\leq 2$. On the other hand $\lambda_\nu(L)>1$ by \cite[Proposition 4.1]{cohn66}. Thus $\lambda_\nu(L)=2$.\\ Assume there is a filtration $\nu'$ on $L$ such that $\lambda_{\nu'}(L)\geq 3$. Then $\nu'$ is a valuation by \cite[Proposition 4.1]{cohn66}. Hence $\nu'(1)=\nu'(1\cdot 1)=\nu'(1)+\nu'(1)$ and therefore $\nu'(1)=0$. It follows that $\nu'(x)=0$ for any right invertible element $x$. On the other hand if $\nu'(x)=0$, then the set $\{x,1\}$ is $L$-dependent with respect to $\nu'$. Hence it is strongly so and we get that $x$ is right invertible. But the right invertible elements of $L$ are precisely the elements of $K\setminus\{0\}$ (since $\nu$ is a valuation). Hence we have shown that $\nu'(x)=0$ if and only if $x\in K\setminus\{0\}$. W.l.o.g. assume that $\nu'(\alpha_1^*)\geq \nu'(\alpha_2^*)$. Set $r_{-1}:=\alpha_1^*$ and $r_0:=\alpha_2^*$. By applying an analog of the Euclidean algorithm to $r_{-1}$ and $r_0$ we get elements $q_1,\dots,q_{n},r_1,\dots,r_{n}\in L$, where $n\geq 1$, such that $\nu'(q_1)\geq 0$, $\nu'(q_2),\dots,\nu'(q_{n})>0$, $\nu'(r_0)>\nu'(r_1)>\dots>\nu'(r_{n})=-\infty$ and 
\[
r_{-1}=r_0q_1+r_1,\quad r_0=r_1q_2+r_2,\quad\dots\quad,r_{n-3}=r_{n-2}q_{n-1}+r_{n-1},\quad r_{n-2}=r_{n-1}q_{n}+r_n
\]
(see \cite[pp. 340--341]{cohn63}). We prove by induction on $i$ that 
\begin{equation}\label{growth}
\nu(r_i)=1+\sum\limits_{j=1}^i\nu(q_j)\text{ for any }i\in\{0,\dots,n\}
\end{equation}
(which means that $\nu(r_i)$ increases as $i$ increases while $\nu'(r_i)$ decreases). One checks easily that (\ref{growth}) holds for $i=0,1$. Let now $2\leq i\leq n$. We have $\nu(r_i)=\nu(r_{i-2}-q_ir_{i-1})$. By the induction hypothesis, $\nu(r_{i-2})= 1+\sum\limits_{j=1}^{i-2}\nu(q_j)$ and $\nu(q_ir_{i-1})=\nu(q_i)+\nu(r_{i-1})=1+\sum\limits_{j=1}^{i}\nu(q_j)$. But $\nu'(q_i)>0$ since $i\geq 2$. Hence $q_i\not\in K$ and therefore $\nu(q_i)>0$. It follows that $\nu(q_ir_{i-1})>\nu(r_{i-2})$ and hence $\nu(r_i)=\nu(q_ir_{i-1})=1+\sum\limits_{j=1}^i\nu(q_j)$. Therefore (\ref{growth}) holds. It follows that $-\infty=\nu(0)=\nu(r_{n})\overset{(\ref{growth})}{=}1+\sum\limits_{j=1}^{n}\nu(q_j)\geq 1$ and hence we have a contradiction. Thus $\lambda(L)=2$.
\end{proof}

Let $K$ be a field and $(E,\omega)$ an LV-rose such that the minimal weight is $2$, the maximal weight is $l\geq 3$ and the number of structured edges is $l+m$ for some $m>0$. By \cite[Theorem 5.21]{hazrat13} and  the previous theorem, $L_K(E,\omega)$ has module type $(l,m)$ (cf. \cite{vitt62}) and dependence number $2$. Let $n,k\geq 1$. By Example \ref{wlpapp}, \cite[Theorem 5.21]{hazrat13} and \cite[Theorem 5.2]{cohn66}, $L_K(n,n+k)$ has module type $(n,k)$ and dependence number $n$. Hence $L_K(E,\omega)$ cannot be isomorphic to one of Leavitt's algebras $L_K(n,n+k)$. In particular, 
if $(E,\omega)$ is the LV-rose from Example \ref{exlvrose} and $K$ a field, then the domain $L_K(E,\omega)$ is not isomorphic to any of the algebras $L_K(n,n+k)$.

In the next three theorems we show that LV-algebras over domains are prime, semi-primitive and nonsingular rings. We further show that contrary to the case of Leavitt path algebras, they are not graded von Neumann regular.

\begin{theorem}\label{cor6}Let $(E,\omega)$ be an LV-graph and $R$ a domain. Then $L_R(E,\omega)$ is a prime ring.
\end{theorem}
\begin{proof}
Let $a,b\in L_R(E,\omega)\setminus\{0\}$. Choose $u,v\in E^0$ such that $au,vb\neq 0$. It follows from (\ref{1}) in Definition \ref{def2} that $\nu(au),\nu(vb)\geq 0$. Since $(E,\omega)$ is connected, there is a generalised path $p$ such that $s(p)=u$ and $r(p)=v$. Clearly $\nu(p)\geq 0$ since $\nu$ is a local valuation.
% We may assume that $p$ is normal and hence that $p\neq 0$ (clearly one can drop all type I-subwords $\alpha_i^\omega(\alpha_j^w)^*$, since the source of such a subword equals its range; then one can replace type II-subwords $\alpha^*_1\beta_1$ by $\alpha^*_2\beta_2$).
It follows that 
\begin{align*}
&\nu(apb)\\
=&\nu(a(upv)b)\\
=&\nu((au)p(vb))\\
\overset{(\ref{4})}{=}&\nu(au)+\nu(p)+\nu(vb)\geq 0.
\end{align*}
It follows from (\ref{1}) in Definition \ref{def2} that $apb\neq 0$ and thus $L_R(E,\omega)$ is prime.
\end{proof}

\begin{theorem}\label{nonsing}
Let $(E,\omega)$ be an LV-graph and $R$ a domain. Then $L_R(E,\omega)$ is a nonsingular ring. 
\end{theorem}

\begin{proof}
Let $a\in L_R(E,\omega)\setminus\{0\}$. Choose a $v\in E^0$ such that $av\neq 0$. Consider the right ideal $vL_R(E,\omega)$. Then $\ann_r(a) \cap vL_R(E,\omega)=0$. For, if there is $b\in vL_R(E,\omega)$ such that $ab=0$, then condition~(\ref{4}) of Definition \ref{def2} implies $b=0$ (as LV-algebras are ``locally'' domain).  This shows that $\ann_r(a)$, $a\in L_R(E,\omega)\setminus\{0\}$,  is not essential and thus $L_R(E,\omega)$ is right nonsingular. The proof for left nonsingularity is similar. 
\end{proof}

Recall that a ring $A$ is called \emph{von Neumann regular} if for any $a\in A$, there is $b\in A$ such that $aba=a$. If $A$ is a graded ring, then $A$ is called \emph{graded von Neumann regular} if the identity above holds for homogeneous elements. 
It is known that Leavitt path algebras are graded von Neumann regular rings~\cite[Corollary~1.6.17]{hazrat16}.   In contrast we have the following theorem. 

\begin{theorem}\label{vonneumann}
Let $(E,\omega)$ be an LV-graph and $R$ a domain. Then $L_R(E,\omega)$ is not \ep{graded} von Neumann regular.
\end{theorem}
\begin{proof}
Choose an $\alpha\in E^{\st}$. Assume that there is an $a\in L_R(E,\omega)$ such that $\alpha_1 a\alpha_1=\alpha_1$. Set $u:=s(\alpha)$ and $v:=r(\alpha)$. Then clearly $vau\neq 0$ (otherwise $\alpha_1 a\alpha_1=0$). It follows that
\begin{align*}
&\nu(\alpha_1a\alpha_1)\\
=&\nu((\alpha_1v)a(u\alpha_1))\\
=&\nu(\alpha_1(vau)\alpha_1))\\
\overset{(\ref{4})}{=}&\nu(\alpha_1)+\nu(vau)+\nu(\alpha_1)\overset{(\ref{1}),(\ref{2})}{>}\nu(\alpha_1).
\end{align*}
Since this is a contradiction, $L_R(E,\omega)$ is not (graded) von Neumann regular.
\end{proof}

\begin{lemma}\label{freg}
Let $(E,\omega)$ be an LV-graph and $R$ a domain. If $J$ is a nonzero ideal of $L_R(E,\omega)$, then for any $n\in\N$ and $u,v\in E^0$ there is an $a\in J\cap u L_R(E,\omega)v$ such that $\nu(a)> n$.
\end{lemma}

\begin{proof}
Let $J$ be a nonzero ideal of $L_R(E,\omega)$, $n\in\N$ and $u,v\in E^0$. Choose a nonzero element $a'\in J$. Then there are $z_1,z_2\in E^0$ such that $z_1a'z_2\neq 0$. Now it easy to show that there are generalised paths $p$ and $q$ of length $>n$ such that $s(p)=u$, $r(p)=z_1$, $s(q)=z_2$, $r(q)=v$ (note that any vertex must emit or receive an structured edge since $(E,\omega)$ is an LV-graph). Clearly $\nu(p)=|p|,\nu(q)=|q|>n$. Set $a:=pa'q\in J\cap uL_R(E,\omega)v$. Then $\nu(a)=\nu(pa'q)\overset{(\ref{4})}{=}\nu(p)+\nu(z_1a'z_2)+\nu(q)>n$.
\end{proof}

\begin{theorem}
Let $(E,\omega)$ be an LV-graph and $R$ a domain. Then the Jacobson radical of $L_R(E,\omega)$ is zero.
\end{theorem}
\begin{proof}
Suppose the Jacobson radical $J$ of $L_R(E,\omega)$ is not zero. Choose a $v\in E^0$. Then, by Lemma~\ref{freg}, there is an $a\in J\cap v L_R(E,\omega)v$ such that $\nu(a)>0$. Since $a\in J$, $a$ is left quasi-regular, i.e., there is a $b\in L_R(E,\omega)$ such that $b+a=ba$. By multiplying $v$ from the right and from the left one gets $vbv+a=vbva$. Hence we may assume that $b\in vL_R(E,\omega)v$. It follows that
\[\max\{\nu(b),\nu(a)\}\overset{(\ref{3})}{\geq} \nu(b+a)=\nu(ba)\overset{(\ref{4})}{=}\nu(b)+\nu(a).\]
This implies that $\nu(b)=0$ and hence, by (\ref{2}), $b=\lambda v$ for some $\lambda\in R\setminus\{0\}$. It follows that $\lambda v=b=ba-a=(\lambda v)a-a=\lambda a-a=(\lambda-1)a$. But this is a contradiction since $\nu(\lambda v)=0$ but either $\nu((\lambda-1)a)=-\infty$, if $\lambda=1$, or $\nu((\lambda-1)a)=\nu(a)>0$, if $\lambda\neq 1$. Thus the Jacobson radical of $L_R(E,\omega)$ is zero.
\end{proof}

\end{document}